\documentclass[a4paper,12pt]{article}
\usepackage{amssymb,amsmath,amsthm,times,mathrsfs,fullpage}
\usepackage[pdftex]{hyperref}
\hypersetup{plainpages=True, pdfstartview=FitH, bookmarksopen=true,
colorlinks=true,linkcolor=blue,citecolor=blue}
\usepackage{tikz,pgf}
\usetikzlibrary{patterns,arrows,decorations.pathreplacing}
\usepackage{epstopdf}
\usepackage[nospace,noadjust]{cite}
\allowdisplaybreaks

\begin{document}

\newtheorem{thm}{Theorem}[section]
\newtheorem{cor}[thm]{Corollary}
\newtheorem{lmm}[thm]{Lemma}
\newtheorem{conj}[thm]{Conjecture}
\newtheorem{pro}[thm]{Proposition}
\theoremstyle{definition}\newtheorem{df}[thm]{Definition}
\theoremstyle{remark}\newtheorem{rem}[thm]{Remark}

\title{{\bf The Generalized Zagreb Index for Non-Plane and Plane Recursive Trees}}
\author{Qunqiang Feng \\ Department of Statistics and Finance \\ University of Science and Technology of China \\ Hefei 230026 \\ P. R. China
\and
Michael Fuchs \\ Department of Mathematical Sciences \\ National Chengchi University \\ Taipei 116 \\ Taiwan
\and
Tsan-Cheng Yu\thanks{Corresponding Author, Email: tsancheng@mail.fju.edu.tw}\\
Department of Mathematics\\
Fu Jen Catholic University\\
New Taipei City, 242062\\
Taiwan}
\date{\today}
\maketitle
\begin{abstract}
The Zagreb index, which is defined as the sum of squares of degrees of the nodes of a tree, was studied in previous works by martingale techniques for random non-plane recursive trees and classes of random trees which are close to random plane recursive trees. These techniques are not easily amended to the generalized Zagreb index, which is defined similar but with squares replaced by higher powers. In this paper, we use the moment transfer approach to (i) obtain the first-order asymptotics of moments and to (ii) prove limit laws for the (suitable normalized) generalized Zagreb index for random non-plane and plane recursive trees; for the former, we show that for all higher powers the limit law is normal, for the latter, we show for cubes and fourth powers that its a non-normal law. 
\end{abstract}

\section{Introduction}
The Zagreb indices are among the most extensively studied topological indices in chemical graph theory. 
Initially introduced by Gutman and Trinajsti\'c in the 1970s \cite{GuTr}, 
the first Zagreb index has been widely utilized to predict the physicochemical properties of chemical compounds \cite{ToCo}. 
It has also been extensively applied in quantitative structure-property relationship (QSPR) and 
quantitative structure-activity relationship (QSAR) studies (see, e.g., \cite{DeBa}). 
Recently, researchers have generalized this index to accommodate more complex molecular structures, 
resulting in the development of the generalized Zagreb index (see \cite{LiZh,ZhZh,BeSa} and references therein). 
These generalizations provide enhanced flexibility in modeling molecular structures and their associated properties \cite{FuGu}.

The generalized Zagreb index is formally defined as follows. 
Let $G=(V,E)$ be a graph with vertex set $V$ and edge set $E$.	
For any positive integer $k$, the $k$-th order generalized Zagreb index of $G$ is defined as
\[
 Z_{G}^{(k)}=\sum_{v\in V} D_v^{k}=\sum_{uv\in E}\big(D_u^{k-1}+D_v^{k-1}\big),
\]
where $D_v$ denotes the degree of vertex $v$ in $G$.
Specially, $Z_{G}^{(k)}$ corresponds to the first Zagreb index when $k=2$, 
and it is referred to as the forgotten topological index when $k=3$ \cite{FuGu}. 
The mathematical properties of the generalized Zagreb index have been extensively studied.
For example, extremal values for this index in trees and unicyclic graphs are determined in \cite{LiZh2004,HuDe},
and upper or lower bounds of $Z_{G}^{(k)}$ are derived for some specific classes of graphs, 
such as planar graphs \cite{HaJe}, bipartite graphs \cite{XuTa}, and line graphs \cite{ChBo}.
For a comprehensive review of the first Zagreb index and its applications in chemistry, we refer to \cite{GuDa}, 
and for relations between the generalized Zagreb index and other topological indices, see \cite{AlGu}.

Random non-plane recursive trees (also just called recursive trees) and plane recursive trees (also called plane-oriented recursive trees or PORTs) are two fundamental structures in the study of random graphs and combinatorial probability (see, e.g., \cite{Dr}).
Using martingale techniques, the limit laws for the first Zagreb index of non-plane recursive trees and 
Barab\'asi-Albert trees (a slight variant of plane recursive trees) are established in \cite{FeHu,FeHu2013};
see also \cite{ZhP} for explicit expressions of the first two moments and the limit law of the first Zagreb index for another (slight) variant of plane recursive trees. 
Furthermore, with an application of Stein's method,
the asymptotic normality of the first Zagreb index for classical Erd\H{o}s-R\'enyi (ER) random graphs is established 
as the graph size approaches infinity \cite{FHS}. 
For the generalized Zagreb index of ER random graphs,  
its expectation is obtained in \cite{DoHo}, and several limits laws are further demonstrated in \cite{FeRe}. 
Motivated by these studies, our work focuses on the asymptotic behavior of this index for non-plane and plane recursive trees.  
While extending prior work \cite{FeHu,FeHu2013,ZhP} might seem natural, 
the martingale techniques and Stein's method
are not applicable for general $k$. 
Instead, we employ the moment transfer approach in this work.

The rest of this paper is organized as follows. In the next section, we formally define random non-plane and plane recursive trees, obtain distributional recurrences for the general Zagreb index, and recall the moment transfer approach. In Section~\ref{nprt}, we apply this method to derive the first-order asymptotics of moments and a central limit theorem for non-plane recursive trees (including the case $k=2$ which was already obtained in earlier works). In Section~\ref{prt}, we show corresponding results for plane recursive trees (for $k\geq 3$ as the case $k=2$ requires a different treatment). We end the paper with a conclusion in Section~\ref{con}.

\section{Recurrences and the Moment-Transfer Approach}

\paragraph{Random non-plane and plane recursive trees.} We first give precise definitions of the two random tree models considered in this paper.

First, a {\it random non-plane recursive tree} of size $n$ is recursively built as follows. Start with a root node labeled by $1$. In the $i$-th step, a node labeled by $i$ joins to already built tree by becoming the child of one of $i-1$ existing nodes, where the parent node is picked uniformly at random from these $i-1$ nodes. We stop when exactly $n$ nodes have joined.

Note that the children of every non-leaf node in a random non-plane recursive tree are not ordered. On the other hand, in a {\it random plane recursive tree}, in the $i$-th step, we assume that a node with $k$ children has $k+1$ free places (one in front of the first child, the next between the first and the second child, etc.). The $i$-th node then chooses uniformly at random from one of those free places (among all the free places) and is attached there. Thus, in contrast to random non-plane recursive trees, the children of every non-leaf node in a random plane recursive tree have a left-to-right order.

The above definitions easily imply that there are $(n-1)!$ different non-plane recursive trees of size $n$ and
\[
(2n-3)!!=\frac{(2n-2)!}{2^{n-1}(n-1)!}=\frac{n!C_{n}}{2^{n-1}}
\]
different plane recursive trees of size $n$, where $C_n$ denotes the (shifted) Catalan numbers. Moreover, from this, we can re-derive the distribution of the size $I_n$ of the left-most subtree of the root of a random non-plane or plane recursive tree (where for non-plane recursive trees, we order the subtrees of the root according to increasing labels of their roots); see \cite{Dr} and \cite{Hw}.

\begin{lmm}\label{dis-In}
\begin{itemize}
\item[(i)] For non-plane recursive trees, $I_n$ has a uniform distribution on the set $\{1,\ldots,n-1\}$.
\item[(ii)] For plane recursive trees, 
\[
{\mathbb P}(I_n=j)=\frac{2(n-j)C_jC_{n-j}}{nC_n},\qquad (1\leq j\leq n-1).
\]
\end{itemize}
\end{lmm}
\begin{proof}
Both results follow from (simple) counting arguments using the above numbers of non-plane and plane recursive trees and decomposing the tree into the left-most subtree of the root and the remaining tree (both of which are again non-plane and plane recursive trees, respectively). First, for part (i):
\[
{\mathbb P}(I_n=j)=\binom{n-2}{j-1}\frac{(j-1)!(n-j-1)!}{(n-1)!}=\frac{1}{n-1},
\]
where the binomial coefficient takes care of the relabeling of the two trees. Likewise for part (ii):
\[
{\mathbb P}(I_n=j)=\binom{n-1}{j}\frac{(2j-3)!!(2n-2j-3)!!}{(2n-3)!!}
\]
which simplifies to the claimed form.
\end{proof}

\paragraph{Generalized Zagreb index and root degree.} The decomposition of random non-plane and plane recursive trees of size $n$ used in the proof of Lemma~\ref{dis-In} can also be used to recursively compute the generalized Zagreb index, which we denote by $Z_n^{(k)}$ throughout this work.

\begin{pro}
The generalized Zagreb index of a random (non-plane or plane) recursive tree of size $n$ satisfies, for $n\geq 2$,
\begin{equation}\label{rec-Zn}
Z_{n}^{(k)}\stackrel{d}{=}Z_{I_n}^{(k)}+\tilde{Z}_{n-I_n}^{(k)}-R_{I_n}^k+(R_{I_n}+1)^k-\tilde{R}_{n-I_n}^k+(\tilde{R}_{n-I_n}+1)^k
\end{equation}
with initial condition $Z_1^{(k)}=0$. Here, $R_n$ denotes the (random) degree of the root, $\tilde{Z}_n$ and $\tilde{R}_n$ are independent copies of $Z_n$ and $R_n$, and $(Z_n,R_n)$ and $(I_n)$ are independent. 
\end{pro}

\begin{proof} The decomposition of a random (non-plane or plane) recursive tree into the left-most subtree of the root and the remaining tree gives two independent (non-plane or plane) recursive trees of size $I_n$ and $n-I_n$. Also, the Zagreb index is then the sum of Zagreb indices of these two trees where we have to correct the contributions at the two roots. For the root of the left-most subtree this correction gives the term $-R_{I_n}^k+(R_{I_n}+1)^k$ since we are missing one edge of the root. Likewise, for the remaining tree, we have the correction $-\tilde{R}_{n-I_n}^k+(\tilde{R}_{n-I_n}+1)^k$.
\end{proof}

The same argument can be used to obtain a recurrence for the (random) root degree, too.

\begin{lmm}
The degree of the root of a random (non-plane or plane) recursive tree of size $n$ satisfies, for $n\geq 2$,
\begin{equation}\label{rec-Rn}
R_n\stackrel{d}{=}\tilde{R}_{n-I_n}+1
\end{equation}
with initial condition $R_1=0$.
\end{lmm}

From this result, we see that all (non-centered or centered) moments of $R_n$ satisfy a recurrence of the type:
\begin{equation}\label{one-sided}
a_n=\sum_{j=1}^{n-1}\pi_{n,j}a_{n-j}+b_n,\qquad (n\geq 2)
\end{equation}
with initial condition $a_1=0$, where $\pi_{n,j}={\mathbb P}(I_n=j)$ and $b_n$ is a function which involves moments of lower order. By solving this recurrence (either exactly or asymptotically), a great deal of properties can be proved for the root degree; see Section~6.1.1 in \cite{Dr} for random non-plane recursive trees and \cite{Hw} for random plane recursive trees. We gather some of these properties (which are used below) in the next proposition.

\begin{pro}\label{ll-dis-Rn}
\begin{itemize}
\item[(i)] For random non-plane recursive trees, in distribution and with convergence of all moments, 
\[
\frac{R_n-\log n}{\sqrt{\log n}}\stackrel{d}{\longrightarrow} N(0,1),\qquad (n\rightarrow\infty),
\]
where $N(0,1)$ denotes a standard normal distribution.
\item[(ii)] For random plane recursive trees, in distribution and with convergence of all moments,
\[
\frac{R_n}{\sqrt{n}}\stackrel{d}{\longrightarrow}{\rm Rayleigh}(\sqrt{2}),
\]
where ${\rm Rayleigh}(\sigma)$ denotes a Rayleigh distribution with parameter $\sigma$.
\end{itemize}
\end{pro}

Similarly, all (non-centered and centered) moments of $Z_n^{(k)}$ satisfy the {\it two-sided} version of~(\ref{one-sided}):
\begin{equation}\label{two-sided}
a_n=\sum_{j=1}^{n-1}\pi_{n,j}(a_j+a_{n-j})+b_n,\qquad (n\geq 2)
\end{equation}
with initial condition $a_1=0$, where $b_n$ is again a function of lower-order moments and moments of $R_n$. 

This recurrence will be the starting point of our analysis of the generalized Zagreb index. To derive the first-order asymptotics of moments and limit laws from it, we are going to apply the {\it moment-transfer approach} which we explain next. (Similarly, one could re-prove Proposition~\ref{ll-dis-Rn} with this method from (\ref{one-sided}); see \cite{ChFu}.)

\paragraph{Moment-transfer approach.} Our goal is to prove limit laws for the generalized Zagreb index of non-plane and plane recursive trees. One way of proving a limit law in probability theory is to use the {\it method of moments}; see Section~30 in \cite{Bi}. This method proceeds by (i) finding the first-order asymptotics of all moments (of the suitably normalized sequence of random variables) and (ii) identifying the limit law from the asymptotic moment sequence. The method is in particular well-suited for sequences of random variables which satisfy distributional recurrences such as (\ref{rec-Zn}) or (\ref{rec-Rn}) since all moments satisfy the same type of recurrence (see (\ref{two-sided}) for the former and (\ref{one-sided}) for the latter) which depends on the sequence $b_n$ that involves lower-order moments. Thus, induction can be applied to find the asymptotics of all moments as follows: first the induction hypothesis is used to obtain the asymptotics of $b_n$. Then, one needs (general) results which bridge the asymptotics of $b_n$ with that of $a_n$; such results are called {\it asymptotic transfer results}. These are then used to complete the induction step. In summary, this method is called the {\it moment-transfer approach}. It has been used in numerous papers analyzing parameters of random trees; see for instance \cite{ChFu,Hw,HwNe} and references therein.

Note that the distributional recurrence (\ref{rec-Zn}) of the Zagreb index depends also on $R_n$. Thus, in the induction step, we have to work with mixed moments of $Z_n^{(k)}$ and $R_n$. Consequently, we need an asymptotic transfer result not only for (\ref{two-sided}) but also for (\ref{one-sided}). Such results for both non-plane and plane recursive trees have been established in earlier work. We will state next the results which we need in Section~\ref{nprt} and Section~\ref{prt}. 

First, for non-plane recursive trees, we need a further notation. Let the (modified) $\hat{\mathcal O}$-notation be defined as the ${\mathcal O}$-notation but with logarithmic terms suppressed, e.g., ${\mathcal O}(\log^kn)=\hat{\mathcal O}(1)$ for any integer $k\geq 0$. Then, we have the following asymptotic transfer results for non-plane recursive trees; see Lemma~2 and Lemma~6 in \cite{HwNe}.

\begin{lmm}\label{at-nprt-os}
Consider (\ref{one-sided}) for non-plane recursive trees. If $b_n=\hat{\mathcal O}(n^{\alpha})$ for $\alpha\geq 0$, then $a_n=\hat{\mathcal O}(n^{\alpha})$.
\end{lmm}

\begin{lmm}\label{at-nprt-ts}
Consider (\ref{two-sided}) for non-plane recursive trees.
\begin{itemize}
\item[(i)] If $b_n=\hat{\mathcal O}(n^{\alpha})$ with $0\leq \alpha<1$, then $a_n=\mu n+\hat{\mathcal O}(n^{\alpha})$ where $\mu\in{\mathbb R}$;
\item[(ii)] If $b_n\sim cn^{\alpha}$ with $\alpha>1$, then $a_n\sim c(\alpha+1)n^{\alpha}/(\alpha-1)$.
\end{itemize}
\end{lmm}

\begin{rem}
The constant in part (i) of the above lemma is given by:
\begin{equation}\label{exp-mu}
\mu=2\sum_{j=2}^{\infty}\frac{b_j}{j(j+1)}.
\end{equation}
Note that, under the assumption of part (i), the series converges.
\end{rem}

Likewise, we have similar results for plane recursive trees; see \cite{Hw}.

\begin{lmm}\label{at-prt-os}
Consider (\ref{one-sided}) for plane recursive trees. If $b_n\sim cn^{\alpha}$ with $\alpha>-1/2$, then $a_n\sim c\Gamma(\alpha+1/2)n^{\alpha+1/2}/\Gamma(\alpha+1)$, where $\Gamma(z)$ denotes the gamma function.
\end{lmm}
\begin{lmm}\label{at-prt-ts}
Consider (\ref{two-sided}) for plane recursive trees.
\begin{itemize}
\item[(i)] If $b_n\sim c\sqrt{n}$ then $a_n\sim cn\log n/\sqrt{\pi}$;
\item[(ii)] If $b_n\sim cn^{\alpha}$ with $\alpha>1/2$, then $a_n\sim c\Gamma(\alpha-1/2)n^{\alpha+1/2}/\Gamma(\alpha)$.
\end{itemize}
\end{lmm}

\section{Non-plane Recursive Trees}\label{nprt}

In this section, we prove a central limit theorem for the generalized Zagreb index of random non-plane recursive trees for all powers. We start by stating the main result of this section.

\begin{thm}\label{main-result-1}
For random non-plane recursive trees, we have the following limit distribution result for the generalized Zagreb index: for $k\geq 2$,
\[
\frac{Z_n^{(k)}-\mu_k n}{\sigma_k\sqrt{n}}\stackrel{d}{\longrightarrow} N(0,1),\qquad (n\rightarrow\infty),
\]
where $\mu_k$ and $\sigma_k$ are positive constants which are given in (\ref{exp-mean}) and the proof of Lemma~\ref{pos-variance}.
\end{thm}

\begin{rem}
As mentioned in the introduction, the limit law for $k=2$ (first Zagreb index) was already obtained in \cite{FeHu}. Our method works for this case, too, and gives an alternative proof.
\end{rem}

We are going to prove the above result by computing the asymptotics of all moments. We first derive the mean and then consider higher moments. The starting point is (\ref{rec-Zn}) which we bring into the form:

\begin{equation}\label{rec-Zn-mod}
Z_n^{(k)}\stackrel{d}{=}Z_{I_n}^{(k)}+\tilde{Z}_{n-I_n}^{(k)}+\sum_{\ell=0}^{k-1}\binom{k}{\ell}(R_{I_n}^{\ell}+\tilde{R}_{n-I_n}^{\ell}).
\end{equation}

\paragraph{Mean value.} By taking expectation on both sides of (\ref{rec-Zn-mod}), we see that the mean $a_n:={\mathbb E}(Z_n^{(k)})$ satisfies (\ref{two-sided}) with 
\begin{equation}\label{mean-nprt-bn}
b_n=\sum_{\ell=0}^{k-1}\binom{k}{\ell}\frac{2}{n-1}\sum_{j=1}^{n-1}{\mathbb E}(R_j^{\ell}).
\end{equation}
Next, by moment convergence in Proposition~\ref{ll-dis-Rn}-(i):
\begin{equation}\label{mom-Rn}
{\mathbb E}(R_n-\log n)^m=\hat{{\mathcal O}}(1),\qquad (m\geq 0),
\end{equation}
which in turn implies that ${\mathbb E}(R_n^{m})=\hat{{\mathcal O}}(1)$ for $ m\geq 0$. Plugging this into (\ref{mean-nprt-bn}) gives $b_n=\hat{\mathcal O}(1)$. Thus, from Lemma~\ref{at-nprt-ts}-(i):
\begin{equation}\label{exp-mean}
\mu(n):={\mathbb E}(Z_n^{(k)})=\mu_kn+\hat{\mathcal O}(1),
\end{equation}
where $\mu_k$ is a suitable constant given by (\ref{exp-mu}). (This also shows that $\mu_k>0$ as $b_n>0$ for $n\geq 3$.)

For small values of $k$, the constant $\mu_k$ can be computed. For example, for $k=2$, the above expression (\ref{mean-nprt-bn}) for $b_n$ becomes
\[
b_n=2+\frac{4}{n-1}\sum_{j=1}^{n-1}{\mathbb E}(R_j).
\]
Note that ${\mathbb E}(R_n)=H_{n-1}$, where $H_{n-1}:=\sum_{j=1}^{n-1}(1/j)$ is the $(n-1)$-st harmonic number. This can easily be derived from (\ref{rec-Rn}); see also Section 6.1.1 in \cite{Dr}. Thus,
\[
b_n=2+\frac{4}{n-1}\sum_{j=1}^{n-1}H_{j-1}=4\Psi(n)+4\gamma-2,
\]
where $\gamma$ is Euler's constant and $\Psi(n)$ denotes the digamma function. Plugging this into (\ref{exp-mu}) gives:
\[
\mu_2=8\sum_{j=2}^{\infty}\frac{\Psi(j)}{j(j+1)}+(2\gamma-4)\sum_{j=2}^{\infty}\frac{1}{j(j+1)}=6,
\]
where the last step follows from
\[
\sum_{j=2}^{\infty}\frac{1}{j(j+1)}=\frac{1}{2}\qquad\text{and}\qquad\sum_{j=2}^{\infty}\frac{\Psi(j)}{j(j+1)}=1-\frac{\gamma}{2}.
\]
This value is the same as the one obtained in \cite{FeHu}. 

\paragraph{Higher moments.} We next consider the asymptotics of higher moments, where we consider central moments.  Therefore, set
\[
\bar{Z}_n^{(k)}:=Z_n^{(k)}-\mu(n).
\]
Then, (\ref{rec-Zn-mod}) becomes
\begin{equation}\label{rec-Zn-shift}
\bar{Z}_n^{(k)}\stackrel{d}{=}\bar{Z}_{I_n}^{(k)}+\tilde{\bar{Z}}_{n-I_n}^{(k)}+\sum_{\ell=0}^{k-1}\binom{k}{\ell}(R_{I_n}^{\ell}
+\tilde{R}_{n-I_n}^{\ell})+\Delta(n),
\end{equation}
where
\[
\Delta(n):=\mu(I_n)+\mu(n-I_n)-\mu(n).
\]
Note that from (\ref{exp-mean}), for $1\leq j\leq n-1$,
\begin{align}
(\Delta(n)\vert I_n=j)&=\mu(j)+\mu(n-j)-\mu(n)\nonumber\\
&=\mu_kj+\mu_k(n-j)-\mu_k n+\hat{\mathcal O}(1)=\hat{\mathcal O}(1).\label{est-Delta}
\end{align}

Since the recurrence (\ref{rec-Zn-shift}) depends on $R_n$, we need to consider mixed moments. We first make the following important observation. (From now on, we suppress the dependence on $k$.)

\begin{lmm}\label{obs}
The mixed moments ${\mathbb E}(\bar{Z}_n^rR_n^s)$ for $s>0$ satisfy (\ref{one-sided}) and for $s=0$ satisfy (\ref{two-sided}), where $b_n$ is in both cases a function of $(\Delta(n)\vert I_n=j)$ and mixed moments of smaller order with respect to the lexicographic order.  
\end{lmm}
\begin{proof}
Raising (\ref{rec-Zn-shift}) to the $r$-th power gives:
\begin{align*}
{\mathbb E}(\bar{Z}_n^r)&=\frac{1}{n-1}\sum_{j=1}^{n-1}{\mathbb E}\left(\bar{Z}_{j}+\tilde{\bar{Z}}_{n-j}+\sum_{\ell=0}^{k-1}\binom{k}{\ell}(R_j^{\ell}
+\tilde{R}_{n-j}^{\ell})+(\Delta(n)\vert I_n=j)\right)^r\\
&=\frac{1}{n-1}\sum_{j=1}^{n-1}\sum_{i_1+i_2+i_3=r}\binom{r}{i_1,i_2,i_3}{\mathbb E}(\bar{Z}_j^{i_1}\tilde{\bar{Z}}_{n-j}^{i_2}T_{n,j}^{i_3}),
\end{align*}
where
\[
T_{n,j}:=\sum_{\ell=0}^{k-1}\binom{k}{\ell}(R_j^{\ell}
+\tilde{R}_{n-j}^{\ell})+(\Delta(n)\vert I_n=j).
\]
Likewise, for the mixed moments, by (\ref{rec-Rn}):
\begin{equation}\label{rec-mixed-moments}
{\mathbb E}(\bar{Z}_n^rR_n^s)=\frac{1}{n-1}\sum_{j=1}^{n-1}\sum_{i_1+i_2+i_3=r}\binom{r}{i_1,i_2,i_3}{\mathbb E}(\bar{Z}_j^{i_1}\tilde{\bar{Z}}_{n-j}^{i_2}T_{n,j}^{i_3}(\tilde{R}_{n-j}+1)^s).
\end{equation}
Expanding $T_{n,j}^{i_3}$ and $(\tilde{R}_{n-j}+1)^s$ in the above mean gives terms of the form
\begin{equation}\label{terms}
{\mathbb E}(\bar{Z}_j^{i_1}R_j^{i}){\mathbb E}(\bar{Z}_{n-j}^{i_2}R_{n-j}^{\ell}),
\end{equation}
where $i\leq i_3(k-1)$, $\ell\leq i_3(k-1)+s$, and $i_1+i_2+i_3=r$. With respect to the lexicographic order, a highest-order term is obtained if we set $i_2=r$ and $\ell=s$ which gives for (\ref{terms}) ${\mathbb E}(\bar{Z}_{n-j}^r R_{n-j}^{s})$ since $i_1=i_3=0$ and thus $i=0$. Also note that the coefficient of this term is $1$. If $s>0$, the orders of all other terms are strictly smaller with respect to the lexicographic order which gives the first claim. On the other hand, for $s=0$, we have another highest-order term, namely, by setting $i_1=r$ and thus $i_2=i_3=0$ which in turn gives $i=\ell=0$, we obtain ${\mathbb E}(\bar{Z}_j^{r})$ and thus the recurrence in this case is indeed two-sided. 
\end{proof}

We now use this lemma and induction to obtain the following proposition.

\begin{pro}\label{mom-pump}
For $r,s\geq 0$,
\[
{\mathbb E}(\bar{Z}_n^r R_n^s)=\hat{\mathcal O}(n^{r/2}).
\]
Moreover, for $s=0$, we have the refinement:
\begin{equation}\label{first-order-asymp}
{\mathbb E}(\bar{Z}_n^r)\sim g_r\sigma_k^{r}n^{r/2},
\end{equation}
where $\sigma_k$ is a suitable constant and
\[
g_r=\begin{cases}{\displaystyle\frac{(2m)!}{2^mm!}}, &\text{if}\ r=2m;\\[5pt] 0,&\text{if}\ r=2m+1.\end{cases}
\]
\end{pro}
\begin{proof}
We use induction on $(r,s)$, which we again equip with the lexicographic order.

First, the claims hold for $r=0$ and all $s\geq 0$ due to (\ref{mom-Rn}). Moreover, from the definition of $\bar{Z}_n$, the claims also hold for $(r,s)=(1,0)$.

Next, we assume that the claims hold for all $(r',s')$ which are smaller than $(r,s)$ with respect to the lexicographic order, i.e., either $r'<r$ or $r=r'$ and $s'<s$. In order to prove the claims for $(r,s)$, we distinguish the two cases $s>0$ and $s=0$.

If $s>0$, by Lemma~\ref{obs}, we have that ${\mathbb E}(\bar{Z}_n^r R_n^s)$ satisfies (\ref{one-sided}) with $b_n$ being a function of $(\Delta(n)\vert I_n=j)$ and mixed moments of smaller order. Thus, (\ref{est-Delta}) and the induction hypothesis yield $b_n=\hat{\mathcal O}(n^{r/2})$. From this, by  Lemma~\ref{at-nprt-os}, we obtain that ${\mathbb E}(\bar{Z}_n^r R_n^s)=\hat{\mathcal O}(n^{r/2})$. This proves the claim in this case.

If $s=0$, then again by Lemma~\ref{obs}, we have that ${\mathbb E}(\bar{Z}_n^r)$ satisfies (\ref{two-sided}) with $b_n$ given by
\begin{equation}\label{toll-s=0}
b_n=\sum_{i=1}^{r-1}\binom{r}{i}\frac{1}{n-1}\sum_{j=1}^{n-1}{\mathbb E}(\bar{Z}_j^i){\mathbb E}(\bar{Z}_{n-j}^{r-i})+\hat{\mathcal O}(n^{(r-1)/2}),
\end{equation}
where this expression arises from (\ref{rec-mixed-moments}), where we have separated the terms with $i_3=0$ (without the two main terms) and the remaining terms which satisfy the claimed bound by (\ref{est-Delta}) and the induction hypothesis. 

Now, for $r=2$, the first term on the right-hand side of (\ref{toll-s=0}) equals $0$ and thus $b_n=\hat{\mathcal O}(n^{1/2})$. Applying Lemma~\ref{at-nprt-ts}-(i) gives ${\mathbb E}(\bar{Z}_n^2)\sim \sigma_kn$, for a suitable sequence $\sigma_k$, which proves the claim in this case.

Next, for $r>2$, we plug the induction hypothesis into the first term on the right-hand side of (\ref{toll-s=0}) which yields
\begin{align*}
\sum_{i=1}^{r-1}\binom{r}{i}\frac{1}{n-1}\sum_{j=1}^{n-1}{\mathbb E}(\bar{Z}_j^i){\mathbb E}(\bar{Z}_{n-j}^{r-i})&\sim\sigma_k^{r}\sum_{i=1}^{r-1}\binom{r}{i}g_ig_{r-i}\frac{1}{n-1}\sum_{j=1}^{n-1}j^{i/2}(n-j)^{(r-i)/2}\\
&\sim\sigma_k^{r}n^{r/2}\left(\sum_{i=1}^{r-1}\binom{r}{i}g_ig_{r-i}\int_{0}^{1}x^{i/2}(1-x)^{(r-i)/2}{\rm d}x\right),
\end{align*}
where the constant inside the brackets for $r$ odd equals $0$ (since either $g_i=0$ or $g_{r-i}=0$ for $1\leq i\leq r-1$) and for $r=2m$ equals:
\begin{align*}
\sum_{i=1}^{2m-1}\binom{2m}{i}g_ig_{2m-i}\frac{\Gamma(i/2+1)\Gamma((2m-i)/2+1)}{\Gamma(m+2)}&=\sum_{i=1}^{m-1}\binom{2m}{2i}g_{2i}g_{2m-2i}\frac{i!(m-i)!}{(m+1)!}\\
&=\frac{(2m)!(m-1)}{2^m(m+1)!}.
\end{align*}
Thus, in both cases, $b_n\sim\sigma_k^{r}n^{r/2}g_r(r/2-1)/(r/2+1)$. Applying Lemma~\ref{at-nprt-ts}-(ii) gives the claimed result also in this case. 
\end{proof}

The last step is to show that $\sigma_k$ from the last proposition is strictly positive.
\begin{lmm}\label{pos-variance}
For all $k\geq 2$, we have $\sigma_k>0$.
\end{lmm}
\begin{proof}
In order to show the claim, we have to revisit the case $(r,s)=(2,0)$ in the proof of the last proposition. 

First, from (\ref{rec-mixed-moments}), we see that ${\mathbb E}(\bar{Z}_n^{2})$ satisfies (\ref{two-sided}) with 
\[
b_n={\mathbb E}(T_{n,I_n}^2)+4{\mathbb E}(\bar{Z}_{I_n}T_{n,I_n})={\mathbb E}(T_{n,I_n}^2)+4\sum_{\ell=0}^{k-1}\binom{k}{\ell}{\mathbb E}(\bar{Z}_{I_n}R_{I_n}^{\ell}).
\]
Since $\sigma_k$ is given by (\ref{exp-mu}), the claimed result follows if (i) $b_n$ is not identical to zero and (ii) $b_n\geq 0$. Item (i) is easy to check and thus, we only need to verify (ii). For this, we prove by induction on $\ell\geq 0$ that  ${\mathbb E}(\bar{Z}_n R_n^{\ell})\geq 0$. 

The claim holds for $\ell=0$. Next, from (\ref{rec-mixed-moments}), we see that ${\mathbb E}(\bar{Z}_n R_n^{\ell})$ satisfies (\ref{one-sided}) with 
\begin{equation}\label{bn-non-neg}
b_n=\sum_{i=0}^{\ell-1}\binom{\ell}{i}{\mathbb E}(\tilde{\bar{Z}}_{n-I_n}\tilde{R}_{n-I_n}^{i})+{\mathbb E}(T_{n,I_n}(\tilde{R}_{n-I_n}+1)^{\ell}).
\end{equation}
Note that the solution of (\ref{one-sided}) with $\pi_{n,j}=1/(n-1)$ is given by
\[
a_n=b_n+\sum_{j=2}^{n-1}\frac{b_j}{j}.
\]
Thus, the claim follows if again $b_n\geq 0$. The first term on the right-hand side of (\ref{bn-non-neg}) is non-negative by induction hypothesis. The second term, we rewrite as:
\[
{\mathbb E}(T_{n,I_n}(\tilde{R}_{n-I_n}+1)^{\ell})=\sum_{i=0}^{k-1}\binom{k}{i}{\mathbb E}(R_{I_n}^{i}+\tilde{R}_{n-I_n}^{i})(\tilde{R}_{n-I_n}+1)^{\ell}+\Delta(n){\mathbb E}(\tilde{R}_{n-I_n}+1)^{\ell}.
\]
Taking expectations on both sides of (\ref{rec-Zn-shift}) gives
\[
\Delta(n)=-2\sum_{i=0}^{k-1}\binom{k}{i}{\mathbb E}(R_{I_n}^i).
\]
Plugging this into the expression above and simplifying yields
\begin{align*}
{\mathbb E}(T_{n,I_n}(\tilde{R}_{n-I_n}+1)^{\ell})&=\sum_{i=0}^{k-1}\binom{k}{i}({\mathbb E}(\tilde{R}_{n-I_n}^{i}(\tilde{R}_{n-I_n}+1)^{\ell})-{\mathbb E}(R_{I_n}^i){\mathbb E}(\tilde{R}_{n-I_n}+1)^{\ell})\\
&=\sum_{i=0}^{k-1}\binom{k}{i}({\mathbb E}(R_{I_n}^{i}(R_{I_n}+1)^{\ell})-{\mathbb E}(R_{I_n}^i){\mathbb E}(R_{I_n}+1)^{\ell}),
\end{align*}
where the last step follows by symmetry. Finally, by the binomial theorem,
\[
{\mathbb E}(R_{I_n}^{i}(R_{I_n}+1)^{\ell})-{\mathbb E}(R_{I_n}^i){\mathbb E}(R_{I_n}+1)^{\ell}=\sum_{j=0}^{\ell}\binom{\ell}{j}({\mathbb E}(R_{I_n}^{i+j})-{\mathbb E}(R_{I_n}^i){\mathbb E}(R_{I_n}^j))
\]
and the expression inside the sum is non-negative because of H\"{o}lder's inequality. Thus, 
\[
{\mathbb E}(T_{n,I_n}(\tilde{R}_{n-I_n}+1)^{\ell})\geq 0,
\]
which in turn implies that the $b_n$ from (\ref{bn-non-neg}) is non-negative. This concludes the induction and the proof of the lemma.
\end{proof}

We can now complete the proof of Theorem~\ref{main-result-1}.

\begin{proof}[Proof of Theorem~\ref{main-result-1}.]
By Lemma~\ref{pos-variance} and (\ref{first-order-asymp}), we obtain that $X_n=\bar{Z}_n/(\sigma_k n^{1/2})$ satisfies
\[
{\mathbb E}(X_n^{r})\longrightarrow g_r,\qquad (n\rightarrow\infty).
\]
Since ${\mathbb E}(N(0,1))^r=g_r$, the method of moments gives $X_n\stackrel{d}{\longrightarrow}N(0,1)$ as claimed.
\end{proof}

\section{Plane Recursive Trees}\label{prt}

This section contains our main result for random plane recursive trees which is the following limit distribution result.

\begin{thm}\label{main-result-2}
For random plane recursive trees, we have the following limit distribution result for the generalized Zagreb index: for $k=3$ or $k=4$,
\[
\frac{Z_n^{(k)}}{n^{k/2}}\stackrel{d}{\longrightarrow}Z^{(k)},
\]
where $Z^{(k)}$ is a random variable that is uniquely characterized by the moment sequence $\{g_{r,0}\}_{r=0}^{\infty}$ which satisfies (\ref{grs-1}) and (\ref{grs-2}).
\end{thm}
\begin{rem}\begin{itemize}\item[(i)] For $k=2$, a different normalization has to be used; see Remark~\ref{ll-k=2}.
\item[(ii)] For $k\geq 5$, the moment-transfer approach cannot be applied as the sequence $\{g_{r,0}\}_{r=0}^{\infty}$ grows too fast to imply that it uniquely characterizes $Z^{(k)}$; see Remark~\ref{no-mom}.
\end{itemize}
\end{rem}

We again compute the asymptotics of all moments by first considering the mean and then higher moments, where for the mean we bring (\ref{rec-Zn}) into the following form:
\begin{equation}\label{rec-Zn-mod-2}
Z_n^{(k)}\stackrel{d}{=}Z^{(k)}_{I_n}+\tilde{Z}^{(k)}_{n-I_n}+\sum_{\ell=0}^{k-1}\binom{k}{\ell}(R_{I_n}^{\ell}-(-1)^{k-\ell}R_n^{\ell}).
\end{equation}
Note that this differs from (\ref{rec-Zn-mod}) since we have used (\ref{rec-Rn}) to replace the last two terms of (\ref{rec-Zn}) by $-(R_n-1)^{k}+R_n^{k}$ before expanding by the binomial theorem; this change will be make the computation of the (asymptotic) mean easier. 
For higher moments we, however, again use (\ref{rec-Zn-mod}).

In addition, we state a technical lemma which is needed below.
\begin{lmm}\label{tech-est}
For $u,v\geq 0$, we have
\begin{equation}\label{est}
\sum_{j=1}^{n-1}\frac{2(n-j)C_jC_{n-j}}{nC_n}j^{u/2}(n-j)^{v/2}=\begin{cases}{\mathcal O}(n^{v/2}),&\text{if}\ u=0;\\[1pt]{\mathcal O}(n^{v/2}\log n),&\text{if}\ u=1; \\[1pt] {\mathcal O}(n^{(u+v-1)/2}),&\text{if}\ u\geq 2.\end{cases}
\end{equation}
\end{lmm}
\begin{proof}
First recall the following asymptotic result for the Catalan numbers:
\begin{equation}\label{asymp-Cn}
C_n\sim\frac{4^{n-1}}{\sqrt{\pi n^3}},\qquad (n\rightarrow\infty).
\end{equation}
By plugging this into (\ref{est}), we obtain for $u\geq 2$,
\begin{align*}
\sum_{j=1}^{n-1}\frac{2(n-j)C_jC_{n-j}}{nC_n}j^{u/2}(n-j)^{v/2}&={\mathcal O}\left(n^{1/2}\sum_{j=1}^{n-1}j^{(u-3)/2}(n-j)^{(v-1)/2}\right)\\
&={\mathcal O}\left(n^{(u+v-1)/2}\int_{0}^{1}x^{(u-3)/2}(1-x)^{(v-1)/2}{\rm d}x\right)
\end{align*}
and the result follows from this since the last integral is finite. On the other hand, for $u=0$ and $u=1$, the integral does not exist and one has to be more careful. More precisely, in these two cases, we replace the sum after the first step by:
\[
\sum_{j=1}^{n-1}j^{(u-3)/2}\left((n-j)^{(v-1)/2}-n^{(v-1)/2}\right)+n^{(v-1)/2}\sum_{j=1}^{n-1}j^{(u-3)/2}.
\]
Then, the divergence issue is resolved and we can argue as above:
\begin{align*}
n^{1/2}\sum_{j=1}^{n-1}&j^{(u-3)/2}(n-j)^{(v-1)/2}\\
&={\mathcal O}\left(n^{(u+v-1)/2}\int_{0}^{1}x^{(u-3)/2}\left((1-x)^{(v-1)/2}-1\right){\rm d}x+n^{v/2}\sum_{j=1}^{n-1}j^{(u-3)/2}\right)\\[5pt]
&=\begin{cases}{\mathcal O}(n^{v/2}),&\text{if}\ u=0;\\[1pt]{\mathcal O}(n^{v/2}\log n),&\text{if}\ u=1.\end{cases}
\end{align*}
This proves the claimed result.
\end{proof}

\paragraph{Mean Value.} Taking the mean value on both sides of (\ref{rec-Zn-mod-2}) shows that $a_n={\mathbb E}(Z_n^{(k)})$ satisfies (\ref{two-sided}) with
\[
b_n=\sum_{\ell=0}^{k-1}\binom{k}{\ell}\left(\sum_{j=1}^{n-1}\frac{2(n-j)C_jC_{n-j}}{nC_n}{\mathbb E}(R_j^{\ell})-(-1)^{k-\ell}{\mathbb E}(R_n^{\ell})\right).
\]
Next, from the moment convergence in Proposition~\ref{ll-dis-Rn}-(ii), we obtain that
\begin{equation}\label{asymp-Rnm}
{\mathbb E}(R_n^{m})\sim\frac{m!\sqrt{\pi}}{\Gamma((m+1)/2)}n^{m/2},\qquad (m\geq 0),  
\end{equation}
and in particular, ${\mathbb E}(R_n^{m})={\mathcal O}(n^{m/2})$ for $m\geq 0$. Thus, by (\ref{est}) (with $u=\ell$ and $v=0$):
\[
b_n\sim k{\mathbb E}(R_n^{k-1})\sim\frac{k!\sqrt{\pi}}{\Gamma(k/2)}n^{(k-1)/2}
\]
and consequently, by Lemma~\ref{at-prt-ts},
\begin{equation}\label{asymp-mean-prt}
{\mathbb E}(Z_n^{(k)})\sim\begin{cases} 2n\log n,&\text{if}\ k=2;\\[5pt]{\displaystyle\frac{2k!\sqrt{\pi}}{(k-2)\Gamma((k-1)/2)}n^{k/2}},&\text{if}\ k\geq 3.\end{cases}
\end{equation}
\begin{rem}
This difference between $k=2$ and $k\geq 3$ is the reason why Theorem~\ref{main-result-2} does not include the case $k=2$. 
\end{rem}

For small $k$, one can also compute closed-form expressions. We demonstrate this for $k=2$; the computation for higher values of $k$ is similar but gets more and more involved. 

First, from (\ref{rec-Rn}) one obtains that
\[
{\mathbb E}(R_n)=\sum_{j=1}^{n-1}\frac{2(n-j)C_jC_{n-j}}{nC_n}{\mathbb E}(R_{n-j})+1,\qquad (n\geq 2),
\]
where ${\mathbb E}(R_1)=0$. This is equivalent to
\[
nC_n{\mathbb E}(R_n)=2\sum_{j=1}^{n-1}C_j(n-j)C_{n-j}{\mathbb E}(R_{n-j})+nC_n.
\]
Set 
\[
A(z):=\sum_{n\geq 1}nC_n{\mathbb E}(R_n)z^n.
\]
Then,
\[
A(z)=2\left(\sum_{n\geq 1}C_nz^n\right)A(z)+\sum_{n\geq 2}nC_nz^n.
\]
Note that
\[
\sum_{n\geq 1}C_nz^n=\frac{1-\sqrt{1-4z}}{2}\qquad\text{and}\qquad\sum_{n\geq 2}nC_nz^n=\frac{z}{\sqrt{1-4z}}-z.
\]
Consequently,
\[
A(z)=\frac{z}{1-4z}-\frac{z}{\sqrt{1-4z}}
\]
and thus,
\begin{equation}\label{closed-Rn}
{\mathbb E}(R_n)=\frac{[z^n]A(z)}{nC_n}=\frac{4^{n-1}}{nC_n}-1,
\end{equation}
where the last step follows from standard computations. (Here, $[z^n]f(z)$ denotes the $n$-th coefficient in the Maclaurin series of $f(z)$.) 

Now, let us turn to ${\mathbb E}(Z_n^{(2)})$ which due to (\ref{rec-Zn-mod}) satisfies
\[
{\mathbb E}(Z_n^{(2)})=2\sum_{j=1}^{n-1}\frac{C_jC_{n-j}}{C_n}{\mathbb E}(Z_j^{(2)})+2+4\sum_{j=1}^{n-1}\frac{C_jC_{n-j}}{C_n}{\mathbb E}(R_j),\qquad (n\geq 2)
\]
with ${\mathbb E}(Z_1^{(2)})=0$. This can be rewritten as
\[
C_n{\mathbb E}(Z_n^{(2)})=2\sum_{j=1}^{n-1}C_{n-j}C_j{\mathbb E}(Z_j^{(2)})+2C_n+4\sum_{j=1}^{n-1}C_{n-j}C_j{\mathbb E}(R_j)
\]
which by setting
\[
B(z):=\sum_{n\geq 1}C_n{\mathbb E}(Z_n^{(2)})z^n
\]
translates into
\[
B(z)=2\left(\sum_{n\geq 1}C_n z^n\right)B(z)+2\sum_{n\geq 2}C_nz^n+4\left(\sum_{n\geq 1}C_nz^n\right)\sum_{n\geq 1}C_n{\mathbb E}(R_n)z^n.
\]
From (\ref{closed-Rn}), we have
\[
\sum_{n\geq 1}C_n{\mathbb E}(R_n)z^n=\frac{1}{4}\log\frac{1}{1-4z}-\frac{1-\sqrt{1-4z}}{2}.
\]
Collecting everything and simplifying gives:
\[
B(z)=\frac{1}{2\sqrt{1-4z}}\log\frac{1}{1-4z}-\frac{2z}{\sqrt{1-4z}}-\frac{1}{2}\log\frac{1}{1-4z}+1-\sqrt{1-4z}.
\]
Recall that
\[
[z^n]\frac{1}{\sqrt{1-z}}\log\frac{1}{1-z}=\frac{1}{4^n}\binom{2n}{n}(2H_{2n}-H_n),
\]
where $H_n=\sum_{j=1}^{n}(1/j)$ denotes the $n$-th Harmonic number. Consequently,
\[
{\mathbb E}(Z_n^{(2)})=(2n-1)(2H_{2n}-H_n)-2n-\frac{4^n}{2nC_n}+2
\]
From this, we obtain the refined asymptotics:
\begin{equation}\label{exp-mean-2}
{\mathbb E}(Z_n^{(2)})=2n\log n+(4\log 2+2\gamma-2)n+{\mathcal O}(\sqrt{n}),\qquad (n\rightarrow\infty),
\end{equation}
where $\gamma$ denotes Euler's constant.

We next consider higher moments.

\paragraph{Higher Moments.} Here, in contrast to the non-plane case, we do not need to shift by the mean value, which simplifies the proof. We drop from now on the superscript $k$.

We start with an analogue of Lemma~\ref{obs}.

\begin{lmm}\label{obs-2}
The mixed moments ${\mathbb E}(Z_n^{r}R_n^{s})$ for $s>0$ satisfy (\ref{one-sided}) and for $s=0$ satisfy (\ref{two-sided}), where $b_n$ is in both cases a function of mixed moment of smaller order with respect to the lexicographic order. 
\end{lmm}
\begin{proof}
Similar as in the proof of Lemma~\ref{obs}, we obtain that
\begin{equation}\label{EZrRs}
{\mathbb E}(Z_n^rR_n^s)=\sum_{j=1}^{n-1}\frac{2(n-j)C_jC_{n-j}}{nC_n}\sum_{i_1+i_2+i_3=r}\binom{r}{i_1,i_2,i_3}{\mathbb E}(Z_j^{i_1}\tilde{Z}_{n-j}^{i_2}T_{n,j}^{i_3}(\tilde{R}_{n-j}+1)^s),
\end{equation}
where
\[
T_{n,j}:=\sum_{\ell=0}^{k-1}\binom{k}{\ell}(R_j^{\ell}+\tilde{R}_{n-j}^{\ell}).
\]
The rest of the arguments are identical to those in the proof of Lemma~\ref{obs}.
\end{proof}

From this result and induction, we obtain the following.

\begin{pro}\label{mixed-moments}
For $r,s\geq 0$, we have
\begin{equation}\label{exp-ZrRs}
{\mathbb E}(Z_n^rR_n^s)\sim g_{r,s}n^{(kr+s)/2}
\end{equation}
where $g_{r,s}$ satisfies the recurrence in the proof below.
\end{pro}
\begin{proof}
We use induction on $(r,s)$ with respect to the lexicographic order, where the claim holds for $r=0$ and arbitrary $s$ because of (\ref{asymp-Rnm}) and for $(r,s)=(1,0)$ because of (\ref{asymp-mean-prt}).

We next assume that the claim holds for all $(r',s')$ which are lexicographically smaller than $(r,s)$. We want to prove it for $(r,s)$.

First consider $s>0$. Then, by Lemma~\ref{obs-2}, we know that ${\mathbb E}(Z_n^{r}R_n^{s})$ satisfies (\ref{one-sided}), where $b_n$ consists of the terms on the right-hand side of (\ref{EZrRs}) except the term with $i_2=r$ and $\tilde{R}_{n-j}^s$ in the expansion of $(\tilde{R}_{n-j}+1)^s$. By expanding, we see that these terms are a linear combination of terms
\begin{equation}\label{terms-bn}
{\mathbb E}(Z_j^{i_1}R_j^{\ell_1}){\mathbb E}(Z_{n-j}^{i_2}R_{n-j}^{\ell_2+\ell_3}),
\end{equation}
where $i_1+i_2+i_3=r$, $\ell_1$ and $\ell_2$ arise from expanding $T_{n,j}^{i_3}$ and thus $\ell_1+\ell_2\leq i_3(k-1)$, and $\ell_3$ arises from expanding $(\tilde{R}_{n-j}+1)^s$ and thus $\ell_3\leq s$. We plug the induction hypothesis into this and then apply Lemma~\ref{tech-est}. Note that whenever $i_1+i_2$ is reduced by $1$, we lose a factor of $k$ in front of $r$ in the exponent of (\ref{exp-ZrRs}). On the other hand, from $\ell_1+\ell_2\leq i_3(k-1)$, we see that we can gain at most a term $k-1$. Thus, by a careful analysis of the terms (\ref{terms-bn}), we see that the following terms dominate:
\begin{itemize}
\item[(i)] $i_1+i_2=r$ with $i_2<r$. Then, $i_3=0$ and thus $\ell_1=\ell_2=0$. Moreover, $\ell_3=s$.
\item[(ii)] $i_2=r$ and thus $i_1=i_3=0$ which in turn implies that $\ell_1=\ell_2=0$. Moreover, $\ell_3=s-1$. (Recall that $\ell_3=s$ is not allowed as this term was removed from $b_n$.)
\item[(iii)] $i_2=r-1$ and $i_1=0, \ell_1=0$ which implies that $\ell_2=k-1$ and $\ell_3=s$.
\end{itemize}
In all remaining cases, the term (\ref{terms-bn}) contributes ${\mathcal O}(n^{(kr+s)/2-1})$. Thus, up to this error, $b_n$ is given by:
\[
\sum_{j=1}^{n-1}\frac{2(n-j)C_jC_{n-j}}{nC_n}\left(\!\sum_{\ell=1}^{r}\binom{r}{\ell}{\mathbb E}(Z_j^{\ell}){\mathbb E}(Z_{n-j}^{r-\ell}R_{n-j}^s)+kr{\mathbb E}(Z_{n-j}^{r-1}R_{n-j}^{k+s-1})+s{\mathbb E}(Z_{n-j}^rR_{n-j}^{s-1})\!\right).
\]
We consider the three terms in this expression separately. For the first one, by (\ref{asymp-Cn}) and the induction hypothesis:
\begin{align*}
\sum_{\ell=1}^{r}\binom{r}{\ell}\sum_{j=1}^{n-1}&\frac{2(n-j)C_jC_{n-j}}{nC_n}{\mathbb E}(Z_j^{\ell}){\mathbb E}(Z_{n-j}^{r-\ell}R_{n-j}^s)\\
&\sim\sum_{\ell=1}^{r}\binom{r}{\ell}g_{\ell,0}g_{r-\ell,s}\frac{\sqrt{n}}{2\sqrt{\pi}}\sum_{j=1}^{n-1}j^{(k\ell-3)/2}
(n-j)^{(k(r-\ell)+s-1)/2}\\[5pt]
&\sim\frac{n^{(kr+s-1)/2}}{2\sqrt{\pi}}\sum_{\ell=1}^{r}\binom{r}{\ell}g_{\ell,0}g_{r-\ell,s}\int_{0}^{1}x^{(k\ell-3)/2}(1-x)^{(k(r-\ell)+s-1)/2}
{\rm d}x\\[5pt]
&=\frac{n^{(kr+s-1)/2}}{2\sqrt{\pi}\Gamma((kr+s)/2)}\sum_{\ell=1}^{r}\binom{r}{\ell}g_{\ell,0}g_{r-\ell,s}\Gamma\left(\frac{k\ell-1}{2}\right)
\Gamma\left(\frac{k(r-\ell)+s+1}{2}\right).
\end{align*}
Next, consider the second term
\begin{align*}
kr\sum_{j=1}^{n-1}\frac{2(n-j)C_jC_{n-j}}{nC_n}{\mathbb E}(Z_{n-j}^{r-1}R_{n-j}^{k+s-1})&\sim 2krg_{r-1,k+s-1}\sqrt{n}\sum_{j=1}^{n-1}C_j4^{-j}(n-j)^{(kr+s)/2-1}\\[4pt]
&\sim2krg_{r-1,k+s-1}n^{(kr+s-1)/2}\sum_{j\geq 1}C_j4^{-j}\\[4pt]
&\sim krg_{r-1,k+s-1}n^{(kr+s-1)/2},
\end{align*}
where the last step follows from 
\[
\sum_{j\geq 1}C_j4^{-j}=\frac{1}{2}.
\]
Likewise,
\begin{align*}
s\sum_{j=1}^{n-1}\frac{2(n-j)C_jC_{n-j}}{nC_n}{\mathbb E}(Z_{n-j}^{r}R_{n-j}^{s-1})&\sim 2sg_{r,s-1}n^{(kr+s-1)/2}\sum_{j\geq 1}C_j4^{-j}\\
&\sim sg_{r,s-1}n^{(kr+s-1)/2}.
\end{align*}
Using Lemma~\ref{at-nprt-ts}-(ii) gives now the claimed result with
\begin{align}
g_{r,s}=&\frac{1}{2\sqrt{\pi}\Gamma((kr+s+1)/2)}\!\sum_{\ell=1}^{r}\!\binom{r}{\ell}g_{\ell,0}g_{r-\ell,s}\Gamma\!\left(\frac{k\ell-1}{2}\!\right)\!\Gamma\!\left(\frac{k(r-\ell)+s+1}{2}\!\right)\nonumber\\
&\quad+\frac{\Gamma((kr+s)/2)}{\Gamma((kr+s+1)/2)}\left(krg_{r-1,k+s-1}+sg_{r,s-1}\right),\qquad (s>0).\label{grs-1}
\end{align}

Finally, for $s=0$, the arguments are similar. First, by Lemma~\ref{obs-2}, ${\mathbb E}(Z_n^{r}R_n^{s})$ satisfies (\ref{two-sided}) with
\[
b_n=\sum_{j=1}^{n-1}\frac{2(n-j)C_jC_{n-j}}{nC_n}\left(\!\sum_{\ell=1}^{r-1}\binom{r}{\ell}{\mathbb E}(Z_j^{\ell}){\mathbb E}(Z_{n-j}^{r-\ell})+kr{\mathbb E}(Z_{n-j}^{r-1}R_{n-j}^{k-1})\!\right)+{\mathcal O}(n^{(kr+s)/2-1}).
\]
Thus, arguing as above, we obtain the claim with
\begin{align}
g_{r,0}=&\frac{1}{\sqrt{\pi}(kr-2)\Gamma((kr-1)/2)}\sum_{\ell=1}^{r-1}\binom{r}{\ell}g_{\ell,0}g_{r-\ell,0}\Gamma
\left(\frac{k\ell-1}{2}\right)\Gamma\left(\frac{k(r-\ell)+1}{2}\right)\nonumber\\
&\qquad+\frac{kr\Gamma(kr/2-1)}{\Gamma((kr-1)/2)}g_{r-1,k-1}.\label{grs-2}
\end{align}
This concludes the proof.
\end{proof}

Before we prove the main result of this section, we still need another technical lemma.

\begin{lmm}\label{unique}
For $k=3$ or $k=4$, there is a unique random vector $(Z,R)$ whose $(r,s)$-th mixed moment equals $g_{r,s}$.
\end{lmm}

\begin{proof} 
In the appendix, we prove that there is a an absolute constant $A$ such that 
\begin{equation}\label{upper-bound}
g_{r,s}\leq A^{kr+s}\sqrt{(kr+s)!},\qquad (r,s\geq 0).
\end{equation} 

Now, in order to establish the result, we only need to show that the two marginal sequences $\{g_{r,0}\}_{r=0}^{\infty}$ and $\{g_{0,s}\}_{s=0}^{\infty}$ uniquely characterize distributions. This is clear for the latter, as it is the (unique) moments sequence of a Rayleigh distribution. 

As for the former, we use Carleman's condition for the Stieltjes moment problem which means that we have to check the divergence of $\sum_{r=0}^{\infty}g_{r,0}^{-1/(2r)}$:
\[
\sum_{r=0}^{\infty}g_{r,0}^{-1/(2r)}\geq A^{-k/2}\sum_{r=0}^{\infty}(kr)!^{-1/(4r)}.
\]
The series on the right hand-hand side diverges if and only if $k=3$ or $k=4$.
\end{proof}
\begin{rem}\label{no-mom}
It is not hard to see that the estimate $g_{r,0}\leq A^{kr}\sqrt{(kr)!}$ is sharp up to the base of the exponential growth term and thus the series $\sum_{r=0}^{\infty}g_{r,0}^{-1/(2r)}$ is indeed convergent for all $k\geq 5$. As a consequence, it is not clear how to apply the method of moments for $k\geq 5$.
\end{rem}

\begin{proof}[Proof of Theorem~\ref{main-result-2}]
From Proposition~\ref{mixed-moments} and Lemma~\ref{unique}, we can conclude that for $k=3$ or $k=4$:
\[
\left(\frac{Z_n}{n^{k/2}},\frac{R_n}{n^{1/2}}\right)\stackrel{d}{\longrightarrow}(Z,R).
\]
From this the result follows.
\end{proof}
\begin{rem}\label{ll-k=2}
For $k=2$, a similar result holds but with a different normalization. More precisely, in this case, we have to shift the mean (see (\ref{exp-mean-2})) as in Section~\ref{nprt} and the result becomes:
\[
\frac{Z_n^{(2)}}{n}-2\log n-(4\log 2+2\gamma-2)\stackrel{d}{\longrightarrow} Z^{(2)},\qquad (n\rightarrow\infty),
\]
where $Z^{(2)}$ is a random variable which is again uniquely characterized by its moment sequence. This can either be proved with the method from this section, or more easily with martingale arguments; see \cite{ZhP} where these arguments are used for a closely related random model.
\end{rem}

\section{Conclusion}\label{con}

Due to its importance in chemistry, the first Zagreb index has been investigated for several random tree models and limit laws have been proved with martingale techniques and Stein's method; see, e.g., \cite{FeHu} and \cite{ZhP}. These techniques do not seem to work for the generalized Zagreb index. In this paper, we used another approach, namely, the {\it moment-transfer approach} to investigate the generalization.

We concentrated on random non-plane and plane recursive trees for which {\it transfer results}, on which the moment-transfer approach rests, have been established in previous work \cite{Hw,HwNe}. For random non-plane trees, we proved that the the generalized Zagreb index has linear mean and variance, and (when properly normalized) satisfies a central limit theorem regardless of the value of $k$ (the main parameter of the generalized Zagreb index; $k=2$ is the classical case). On the other hand, for random plane recursive trees, the mean grows with $k$ and we proved that the limit law is non-normal for $2\leq k\leq 4$. More precisely, we derived the asymptotic moment sequence for all $k\geq 3$ and showed that it characterizes the limit law for $k=3$ and $k=4$. For the mean, our result on the moments also includes the case $k=2$ for which we have a closed-form expression; it corrects a previous result from \cite{ZhP} which was stated for random plane recursive trees even though the author investigated a slightly different random model. In addition, we commented on the limit law of the case $k=2$ which requires a different normalization; see Remark~\ref{ll-k=2}. In this case, our method of proof works, too, but the result is more easily proved by using martingale techniques (that is why we did not provided details). 

\section*{Acknowledgments}

This research was initiated at the 2024 Institute for Mathematical Statistics – Asia-Pacific Rim Meeting (IMS–APRM 2024) which was held in Melbourne, Australia. A large part of the technical work was done while the second author was visiting the first author at the University of Science and Technology of China (USTC) in Hefei, China. He thanks the first author and USTC for hospitality. The second and third author were partially supported by the National Science and Technology Council (NSTC), Taiwan under grants  NSTC-113-2115-M-004-004-MY3 (MF) and NSCT-114-2115-M-030-006 (TCY).

\section*{Appendix}

In this appendix we prove (\ref{upper-bound}).

\begin{pro}\label{res-app}
For the sequence $\{g_{r,s}\}_{r,s=0}^{\infty}$ from Proposition~\ref{mixed-moments}, we have the bound
\[
g_{r,s}\leq A^{kr+s}\sqrt{(kr+s)!},
\]
where $A$ is a suitable large constant.
\end{pro}

We first need two lemmas.
\begin{lmm}\label{tech-lmm-1}
For all integers $a,b\geq 1$, we have
\[
\frac{\Gamma(a/2)\Gamma(b/2)}{\sqrt{\pi}\Gamma((a+b-1)/2)}\leq 1.
\]
\end{lmm}
\begin{proof}
We prove this by induction on $a+b=k$. Note that the claim is true for $k=2$ and $k=3$. Assume that it holds for $k'<k$. We will show it for $a+b=k$. By the functional equation for the Gamma function, we have
\[
\frac{(a/2-1)\Gamma((a-2)/2)\Gamma(b/2)}{\sqrt{\pi}((a+b-1)/2-1)\Gamma((a+b-3)/2)}=
\frac{a-2}{a-2+b-1}\cdot\frac{\Gamma((a-2)/2)\Gamma(b/2)}{\sqrt{\pi}\Gamma((a+b-3)/2)}\leq 1,
\]
where the first factor is trivially bounded by $1$ and the second factor is bounded by $1$ because of the induction hypothesis.
\end{proof}

\begin{lmm}\label{tech-lmm-2}
For $k\geq 2$, we have
\[
\binom{r}{\ell}^2\leq\binom{kr+s}{k\ell},\qquad (0\leq\ell\leq r),
\]
where $r,s\geq 0$ are integers.
\end{lmm}
\begin{proof}
Partition a set of $kr$ objects into $k$ blocks which all have size $r$. Then, clearly
\[
\binom{r}{\ell}^k\leq\binom{kr}{k\ell}
\]
as selecting $\ell$ elements from each block gives a selection of $k\ell$ elements from the original set. Thus,
\[
\binom{r}{\ell}^2\leq \binom{r}{\ell}^k\leq\binom{kr}{k\ell}\leq\binom{kr+s}{k\ell}
\]
as claimed.
\end{proof}

\begin{proof}[Proof of Proposition~\ref{res-app}] We show the claimed bound by using induction on $(r,s)$ which we equip with the lexicographic order. 

First, for $r=0$, we have
\[
g_{0,s}=\frac{s!\sqrt{\pi}}{\Gamma((s+1)/2)}
\]
as these are the moments of the Rayleigh distribution with parameter $\sqrt{2}$. Thus, the claim for the induction base follows by Stirling's formula for the Gamma function. 

Next, we assume that the claim holds for $(r',s')$ which are (lexicographically) smaller than $(r,s)$. We show that it holds for $(r,s)$ by using the recurrences (\ref{grs-1}) and (\ref{grs-2}) where we concentrate on the former as details for the later are similar (and even easier as one term of (\ref{grs-1}) is missing in (\ref{grs-2})).

We can break the right-hand side of (\ref{grs-1}) into three parts according to its three terms. First, for the first part:
\begin{align*}
\frac{1}{2\sqrt{\pi}\Gamma((kr+s+1)/2)}&\sum_{\ell=1}^{r}\binom{r}{\ell}g_{\ell,0}g_{r-\ell,s}\Gamma\left(\frac{k\ell-1}{2}\right)\Gamma\left(\frac{k(r-\ell)+s+1}{2}\right)\\
&\leq\frac{1}{kr+s-1}\sum_{\ell=1}^{r}\binom{r}{\ell}g_{\ell,0}g_{r-\ell,s}\\
&\leq\frac{A^{kr+s}}{kr+s-1}\sum_{\ell=1}^{r}\sqrt{\binom{r}{\ell}^2(k\ell)!(k(r-\ell)+s)!}\\
&\leq\frac{r}{kr+s-1}A^{kr+s}\sqrt{(kr+s)!}\leq\frac{1}{3}A^{kr+s}\sqrt{(kr+s)!},
\end{align*}
where we used Lemma~\ref{tech-lmm-1} for the first estimate, the induction hypothesis for the second estimate, and Lemma~\ref{tech-lmm-2} for the second last estimate. For the second and third part, recall that for $x\geq 1$,
\[
\frac{\Gamma(x/2)}{\Gamma((x+1)/2)}\leq 2x^{-1/2}.
\]
Thus,
\begin{align*}
&\frac{\Gamma((kr+s)/2)}{\Gamma((kr+s+1)/2)}\left(krg_{r-1,k+s-1}+sg_{r,s-1}\right)\\
&\qquad\leq 2(kr+s)^{-1/2}\left(krA^{kr+s-1}\sqrt{(kr+s-1)!}+sA^{kr+s-1}\sqrt{(kr+s-1)!}\right)\\
&\qquad\leq \frac{2}{A}A^{kr+s}\sqrt{(kr+s)!}\leq\frac{2}{3}A^{kr+s}\sqrt{(kr+s)!},
\end{align*}
where the last step follows since we can assume that $A\geq 3$. 

Combining the above two estimates for the first part and the second and third part completes the induction step and hence the proof.
\end{proof}
\end{document}